\documentclass[a4paper,12pt,intlimits,oneside]{amsart}
\usepackage{amsmath,amsthm,amssymb,mathtools}
\usepackage{dsfont}
\usepackage{mathrsfs}
\usepackage[style=alphabetic,maxnames=200,maxalphanames=6,giveninits]{biblatex}
\usepackage[colorlinks=true]{hyperref}
\usepackage[utf8]{inputenc}
\usepackage{esint}
\usepackage{amsthm}
\usepackage{latexsym}
\usepackage{amssymb}
\usepackage{xcolor}
\usepackage{hyperref}
\usepackage{bbold}
\usepackage{enumerate}
\usepackage{varwidth}
\usepackage[shortcuts]{extdash}

\makeatletter
\numberwithin{equation}{section}
\theoremstyle{plain}
        \newtheorem{theorem}{Theorem}[section]

        \newtheorem{remark}[theorem]{Remark}  
         
        \newtheorem*{claim*}{Claim}  
         
\newtheorem*{theorem*}{Theorem}
\newtheorem*{definition*}{Definition}
\newtheorem*{proposition*}{Proposition}
\usepackage{graphicx}
\let\oldmarginpar\marginpar
\renewcommand\marginpar[1]{\-\oldmarginpar[\raggedleft\footnotesize #1]
{\raggedright\footnotesize #1}}
\newcommand \be {\begin{equation}}
\newcommand \ee {\end{equation}}
\newcommand \la \langle
\newcommand \ra \rangle

\renewcommand\div{\text{div\,}}

\newcommand{\R}{\mathbb{R}}

\renewcommand{\d}{\partial} 
\DeclareMathOperator{\di}{div}

\newcommand{\defeq}{\mathop{=}\limits^{\textrm{def}}}
\title[3D stationary Navier--Stokes]{Three-dimensional stationary incompressible inhomogeneous Navier--Stokes equations in the axially symmetric case}
\author{Zihui He}
\address[Z. He]
{Fakultat für Mathematik, Universität Bielefeld, Postfach 100131, 33501 Bielefeld, Germany.}
\email{zihui.he@uni-bielefeld.de}
\date{\today}
\begin{document}
\subjclass[2020]{35Q30, 76D05}
\keywords{Inhomogeneous incompressible Navier--Stokes equations, variable viscosity coefficient}
\maketitle
\begin{abstract}
We show the existence of (a class of) weak solutions to the 
three-dimensional stationary incompressible inhomogeneous Navier--Stokes equations with density-dependent viscosity coefficient in the axially symmetric case. Further symmetric solutions in cylindrical coordinates, spherical coordinates and   Cartesian coordinates are also discussed.    \end{abstract}

\section{Introduction}
The three-dimensional stationary inhomogeneous incompressible \linebreak Navier--Stokes equations are
\begin{equation}\label{SNS3}
\left\{
\begin{aligned}
&\di(\rho u\otimes u)-\di(\mu Su)+\nabla\Pi= f,\quad x\in\Omega\subset\R^3,\\
&\di u=0,\,\di(\rho u)=0.
\end{aligned}
\right.
\end{equation}
The velocity field $u:\Omega\to\R^3$, the density function $\rho:\Omega\to\R_+$ and the pressure $\Pi:\Omega\to\R$ are unknown.
The external force $f:\Omega\to\R^3$ is given. We write $\nabla u=(\d_ju_i)_{1\le i,j\le3}$, $Su=\nabla u+(\nabla u)^T$ and $\frac12 Su$ is the symmetric part of $\nabla u$. We denote  $v\otimes w=(v_iw_j)_{1\le i,j\le3}$ for vectors $v=(v_1,v_2,v_3)^T$ and $w=(w_1,w_2,w_3)^T$. 

The viscosity coefficient $\mu$ depends smoothly on the density function $\rho$,
\begin{equation*}
    \mu=b(\rho),\quad b\in C(\R_+;[\mu_*,\mu^*]) \quad \text{given},
\end{equation*}
where $\mu_*,\mu^*>0$ are lower and upper bounds.   Notice that if $\mu=\nu>0$ is a positive constant, then $\di(\mu Su)=\nu \Delta u$ since $\di u=0$. 

On a bounded domain $\Omega$, we consider the boundary value problem \eqref{SNS3} under the boundary condition 
\begin{equation}\label{u00:R3}
u|_{\d \Omega}=u_0
\end{equation}
satisfying the zero-flux condition
\begin{equation}\label{u0:R3}
 \int_{\partial\Omega}u_0\cdot \vec{n}\,ds=0.
\end{equation}

\textcite{Leray} showed the solvability of the classical stationary incompressible Navier--Stokes equations 
\begin{equation}
\label{CNS}
\left\{
\begin{aligned}
&\di(  u\otimes u)-\nu \Delta u+\nabla\Pi= f,\\
&\di u=0.
\end{aligned}
\right.
\end{equation}
on certain bounded, exterior domains or $\R^3$. There are some works devoted to the asymptotic behavior of Leray's solutions, see for example \cite{ Finn,AMI91}. We mention the celebrated book on 
 stationary fluid flows by \textcite{Galdi}.

However, to our knowledge, less is known about the stationary inhomogeneous Navier--Stokes equations \eqref{SNS3}. %On a two-dimensional simply connected domain, for any solenoidal vector field $u=(u_1,u_2)$ satisfying the zero flux condition \eqref{u0:R3}, there exists a stream function $\Phi$ such that
%$$u=\nabla^\perp \Phi \quad  \text{and} \quad\nabla^\perp=\begin{pmatrix}\d_{x_2}\\-\d_{x_1}\end{pmatrix}.$$ 
In the \textit{constant} viscosity coefficient case 
\begin{equation*}
\left\{
\begin{aligned}
&\di(\rho u\otimes u)-\nu\Delta u+\nabla\Pi= f,\\
&\di u=0,\,\di(\rho u)=0,
\end{aligned}
\right.
\end{equation*}
\textcite{Frolov} showed the existence of weak solutions with the form  \begin{equation*}
    (\rho,u)=(\eta(\Phi),\nabla^\perp\Phi),\quad  \nabla^\perp=\begin{pmatrix}
    \d_{x_2}\\-\d_{x_1}
    \end{pmatrix},
\end{equation*}
where $\Phi$ is the stream function of $u$ and $\eta$ is any given Hölder continuous function.
Under this assumption, the density equation holds immediately:
\begin{equation*}
    \di(\rho u)=\nabla \eta(\Phi)\cdot \nabla^\perp \Phi=0.
\end{equation*}
Later, \textcite{Santos1} generalised this existence result to the case that $\eta$ is only bounded. Concerning the \textit{density-dependent} viscosity coefficients, the author and Liao \cite{HL20} showed 
existence and regularity results for the equations \eqref{SNS3}.  We mention the celebrated book on the \textit{evolutionary}  incompressible inhomogeneous Navier--Stokes equations by \textcite{Lions}. To our best knowledge, there are no existence results for the three-dimensional \textit{stationary} inhomogeneous incompressible Navier--Stokes equations \eqref{SNS3}.

%\subsection*{Organization of this chapter}
%Subsection \ref{R3main} is devoted to stating our main existence result Theorem \ref{thm:R3} for the equation \eqref{SNS3} in the axially symmetric case. Subsection \ref{sys} is devoted to showing more symmetric solutions in cylindrical, spherical and Cartesian coordinates. We sketch the proof of Theorem \ref{thm:R3} in Section \ref{subsec:proof}.

\subsection{Main result}\label{R3main}
We use the cylindrical coordinates $(r,z,\theta)\in[0,\infty)\times \R\times[0,2\pi)$ and write $e_r,e_z$ and $e_\theta$ for the orthogonal basis
    \begin{equation}
    \label{sc}
       e_r=\begin{pmatrix}
         \cos\theta\\\sin\theta\\0
       \end{pmatrix},\quad 
       e_z=\begin{pmatrix}
        0\\0\\1
\end{pmatrix}
,\quad 
  e_\theta=\begin{pmatrix}
       -\sin\theta\\\cos\theta\\0
       \end{pmatrix}.
\end{equation}
We consider the system \eqref{SNS3} on the axially symmetric simply connected domain
\begin{equation}\label{domian:R3}
\Omega=[0,r_1)\times(z_1,z_2)\times[0,2\pi), 
\end{equation}
where $0<r_1<+\infty$ and $-\infty<z_1<z_2<+\infty$.   
The velocity field 
\begin{equation*}
  u=u_re_r+u_\theta e_\theta+u_ze_z
\end{equation*}
is called {\it axially symmetric} if $u_r,u_\theta$ and $u_z$ are independent of $\theta$. We define the functional spaces for axially symmetric functions
\begin{equation*}
\begin{aligned}
   H_\sigma^1(\Omega)=\{v\in H^1(\Omega)\mid v \text{ is axially symmetric}, \di v=0\}.
\end{aligned}
\end{equation*}

The incompressibility condition for a symmetric vector field $u\in  H_\sigma^1(\Omega)$ reads
\begin{equation}
\label{div}
   \div u=\frac{1}{r}\d_r(ru_r)+\frac{1}{r}\d_z(ru_z)=0,\quad r\neq 0.
\end{equation}
If $u$ also satisfies the zero-flux assumption,
%{\color{blue} give explicit form}
then there exists an axially symmetric stream function $\varphi=\varphi(r,z)$ such that  
\begin{equation}
\label{def:phi}
  ru_r=\d_z\varphi,\quad ru_z=-\d_r \varphi.
\end{equation}
We take any fixed scalar function $\eta\in L^\infty(\R;[\rho_*,\rho^*])$ with $\rho_*,\rho^*>0$. Take a density function 
$\rho=\eta(\varphi)$, then the mass conservation law
\begin{equation*}
 \di(\rho u)=\frac{1}{r}\d_r\rho\d_z\varphi-\frac{1}{r}\d_z\rho\d_r\varphi=0   ,\quad r\neq0,
\end{equation*}
holds in the distribution sense.

Our main result is the following.
The proof is similar to the argument in \cite[Subsection 2.1]{HL20}, which is based on Leray's method \cite{Leray} of solving the classical Navier--Stokes equations \eqref{CNS}.
In Section~\ref{subsec:proof}, we sketch the proof and explain the differences to \cite{HL20}.
\begin{theorem}\label{thm:R3}
Let $b\in C(\R;[\mu_\ast,\mu^*])$, $\mu_\ast,\mu^*>0$ and $\eta\in L^\infty(\R;[\rho_*,\rho^*])$, $0<\rho_\ast\leq \rho^\ast$ be  given. Let $\Omega$ be a bounded connected axially symmetric 
domain defined as in \eqref{domian:R3}. Let $u_0\in H^{1/2}_\sigma(\d\Omega)=\{\text{tr}(u) \mid u\in  H^1_\sigma(\R^3)\}$ and $f\in H^{-1}(\Omega;\R^3)$ be axially symmetric functions satisfying \eqref{u0:R3}. Then there exists at least one  axially symmetric weak solution 
\begin{equation*}
(\rho,u)=(\eta(\varphi),\frac{1}{r}\d_z\varphi e_r-\frac{1}{r}\d_r\varphi e_z+u_\theta e_\theta  )\in L^\infty(\Omega)\times H^1_\sigma(\Omega)    
\end{equation*}
of the boundary value problem \eqref{SNS3}--\eqref{u00:R3}, where $\varphi\in H^2(\Omega)$ is an axially symmetric stream function of $u$, in the sense that
$\div(\rho u)=0$ holds in $\Omega$, $u_0=u|_{\partial\Omega}$ is the trace of $u$ on $\partial\Omega$, and the integral identity 
\begin{equation*}
    \frac12\int_{\Omega}\mu Su:Sv\,dx=\int_{\Omega}(\rho u\otimes u):\nabla v\,dx+\int_{\Omega} f\cdot v\,dx
\end{equation*}
holds for all $v\in H_\sigma^1(\Omega)\cap H^1_0(\Omega)$, where$A:B\defeq\sum_{i,j=1}^3 A_{ij}B_{ij}$ for the matrices $A=(A_{ij})_{1\leq i,j\leq 3}$ and $B=(B_{ij})_{1\le i,j\le 3}$.
\end{theorem}
\begin{remark}
\begin{itemize}
    \item The assumption \eqref{domian:R3} on the domain can be relaxed to any $C^{1,1}$-symmetric domain with respect to the coordinate axis. The $C^{1,1}$-regularity is necessary to extend  $u_0\in H^{1/2}(\d\Omega)$ to a $H^1$-regularity function on $\R^3$.
    
    \item  One can generalise the theorem to axially symmetric multi-connected domains $\d\Omega=\bigcup_{i=1}^k \Gamma_i$ and under the assumption that there is no flux through each component   
\begin{equation}
\label{A:R3}
\int_{\Gamma_i}u_0\cdot \vec{n}\,ds=0,\quad i=1,\ldots,k.
\end{equation} 
    \item Following Leray's approximation method \cite{Leray}, one can generalise the existence result to exterior domains and the whole space $\R^3$.
\end{itemize}
\end{remark}

\subsection{Other symmetric solutions}\label{sys}
In this section, we will show the existence of symmetric solutions of \eqref{SNS3}--\eqref{u00:R3} in cylindrical coordinates, spherical coordinates and   Cartesian coordinates.
The key point is to choose the structure of $\rho$ carefully such that the density equation
\begin{equation}\label{MC}
    \di(\rho u)=0
\end{equation}
holds automatically. We will also formulate explicit examples in  Cartesian coordinates.

More precisely, we consider the following two types of symmetries:
\begin{itemize}
    \item \textit{Symmetry type I:} We assume that $u$ is a three-dimensional vector-valued function depending only on a two-dimensional variable, and hence there exists a stream function $\varphi$ of $u$. We take any bounded positive function $\eta$ and
    the density function of the form
    $$\rho=\eta(\varphi).$$
    This is the case in Theorem \ref{thm:R3}.
    \item \textit{Symmetry type II:} We assume that $u$ is a two-dimensional vector-valued function and vanishes in the spatial direction $e_w$ corresponding to the variable $x_w$.  The density function $\rho$ has the form
    $$\rho=\eta(x_w).$$
\end{itemize}
The mass conservation law  \eqref{MC} holds immediately for the above symmetry types.

We recall the coordinate axis $e_r,e_z$ and $e_\theta$ as in  \eqref{sc} in cylindrical coordinates. We denote the unit standard vectors in 
Cartesian coordinates $(x_1,x_2,x_3)\in\R^3$ by 
$$e_1=\begin{pmatrix}
1\\0\\0
\end{pmatrix},\quad
e_2=\begin{pmatrix}
0\\1\\0
\end{pmatrix}
,\quad
e_3=\begin{pmatrix}
0\\0\\1
\end{pmatrix}$$
and in spherical coordinates $(\tilde r,\alpha,\theta)\in [0,\infty)\times[0,\pi]\times[0,2\pi)$ by
    \begin{equation*}
     e_{\tilde r}=\begin{pmatrix}
       \sin\alpha\cos\theta\\\sin\alpha\sin\theta\\\cos\alpha
     \end{pmatrix},\quad
      e_\alpha=\begin{pmatrix}
       \cos\alpha\cos\theta\\\cos\alpha\sin\theta\\-\sin\alpha
     \end{pmatrix},\quad
      e_\theta=\begin{pmatrix}
       -\sin\theta\\\cos\theta\\0
     \end{pmatrix}.
    \end{equation*}

We have the following existence theorem.
\begin{theorem}\label{THM2}
Let $b\in C(\R;[\mu_\ast,\mu^*])$, $\mu_\ast,\mu^*>0$, and $\eta\in L^\infty(\R;[\rho_*,\rho^*])$, $\rho_\ast,\rho^\ast>0$ be  given. {In cylindrical, Cartesian or spherical coordinates, 
let $\Omega$ be a bounded connected symmetric domain
\begin{align*}
& \Omega=[0,r_1)\times(z_1,z_2)\times[0,2\pi),\\
\text{or}\quad&\Omega=(x_1^1,x_1^2)\times(x_2^1,x_2^1)\times(x_3^1,x_3^2),\\
\text{or}\quad&\Omega=[0,\tilde r_1)\times[0,\pi]\times[0,2\pi).
\end{align*} }
%where $0<r_1,\tilde r_1<+\infty$, $-\infty< z_1<z_2<\infty$ and $x_1^1<x_1^2,\, x_2^1<x_2^1,\, x_3^1<x_3^2<+\infty$.}
Let $u_0\in H^{1/2}(\d\Omega)$ satisfy the zero-flux assumption \eqref{A:R3} and let $f\in H^{-1}(\Omega;\R^3)$.
\begin{itemize}
    \item \textit{Symmetry type I:} If  $u_0$ and $f$ only depend only on $r$ and $z$, or $x_1$ and $x_2$, or $\tilde r$ and $\alpha$, then there exists at least one solution in the weak sense as in Theorem \ref{thm:R3},
\begin{equation*}
    (\rho,u)\in L^\infty(\Omega)\times H^1(\Omega)
\end{equation*}
of the form 
\begin{align}
&(\rho,u)=(\eta(\varphi),\frac{1}{r}\d_z\varphi e_r-\frac{1}{r}\d_r\varphi e_z+u_\theta e_\theta),\label{CS}\\
&\text{%
$\varphi$ and $u_\theta$ depending only on
$r$ and $z$},\notag\\
\text{or}\quad & (\rho,u)=(\eta(\varphi),\d_2 \varphi e_1-\d_1\varphi e_2+u_3e_3),\label{ccS}\\
&\text{%
$\varphi$ and $u_3$ depending only on
$x_1$ and $x_2$},\notag\\
\text{or}\quad & (\rho,u)=(\eta(\varphi),\frac{1}{{\tilde r}^2 \sin\alpha}\d_\alpha\varphi e_{\tilde r}-\frac{1}{{\tilde r} \sin\alpha}\d_{\tilde r}\varphi e_\alpha+ u_\theta e_\theta), 
\label{SS}\\
&\text{%
$\varphi$ and $u_\theta$ depending only on
$\tilde r$ and $\alpha$}\notag
\end{align}
 of the boundary value problem \eqref{SNS3}--\eqref{u00:R3}.

    \item \textit{Symmetry type II:} If  $u_0$ and $f$ have the form
   \begin{alignat*}{3}
&u_0&&=u_{0,r}e_r+u_{0,z}e_z,\quad &&f=f_re_r+f_ze_z,\\
\text{or}\quad &u_0&&=u_{0,1}e_1+u_{0,2}e_2,\quad &&f=f_1e_1+f_2e_2,\\
\text{or}\quad
&u_0&&=u_{0,\tilde r}e_{\tilde r}+u_{0,\alpha}e_\alpha,\quad &&f=f_{\tilde r}e_{\tilde r}+f_\alpha e_\alpha.
\end{alignat*}
    Then there exists at least one weak solution 
\begin{equation*}
    (\rho,u)\in L^\infty(\Omega)\times H^1(\Omega)
\end{equation*}
of the form 
\begin{align*}
&(\rho,u)=(\eta(\theta),u_re_r+u_ze_z),\\
\text{or}\quad & (\rho,u)=(\eta(x_3),u_1e_1+u_2e_2),\\
\text{or}\quad & (\rho,u)=(\eta(\theta),u_{\tilde r} e_{\tilde r}+u_\alpha e_\alpha)
\end{align*}
 of the boundary value problem \eqref{SNS3}--\eqref{u00:R3}.

\end{itemize}

\end{theorem}
\begin{proof}
For the symmetric solutions of symmetry type~I, the cylindrical case was shown in Theorem \ref{thm:R3}, analogously we can show the Cartesian and spherical cases. The proof of symmetry type~II is similar to type~I, see the solvability in Remark \ref{rem:R3}.
We omit the detailed proof here and only verify the mass conservation law. 

In cylindrical, Cartesian, and spherical coordinates, the gradient operators can be written as
\begin{align*}
\nabla&=e_r\d_r+e_z\d_z+\frac{e_\theta}{r}\d_\theta,\\
\nabla&=e_1\d_1+e_2\d_2+e_3\d_3,\\
\nabla & =  e_{\tilde r}\,\d_{\tilde r}+\frac{e_\alpha\,}{\tilde r}\,\d_\alpha+\frac{e_\theta}{\tilde r \sin\alpha}\d_\theta.
\end{align*}
For the solutions of symmetry type II, the mass conservation law holds immediately since
\begin{equation*}
    \div(\rho u)=\nabla \rho\cdot u.
\end{equation*}
Concerning the solutions of symmetry type I, the case of cylindrical coordinates was shown in Section \ref{R3main}.  In Cartesian and spherical coordinates we consider 
\begin{equation*}
    u=u_1e_1+u_2e_2+u_3e_3,\quad u=u_{\tilde r} e_{\tilde r}+u_\alpha e_\alpha+u_\theta e_\theta,
\end{equation*}
where $u_1,u_2$ and $u_3$ depend only on $x_1$ and $x_2$, and $u_{\tilde r},u_\alpha$ and $u_\theta$ depend only on $\tilde r$ and $\alpha$. Then the incompressibility conditions can be written as
\begin{equation}\label{div3}
\begin{aligned}
\div u&=\d_1u_1+\d_2u_2=0, \\
\div u&=\frac{1}{{\tilde r}^2}\d_{\tilde r}({\tilde r}^2u_{\tilde r})+\frac{1}{{\tilde r} \sin\alpha}\d_\alpha(\sin\alpha u_\alpha)=0.
\end{aligned}
\end{equation}
If $u$ satisfies \eqref{div3} and the zero-flux assumption \eqref{u00:R3}, then there exists stream functions $\varphi=\varphi(x_1,x_2)$ and
$\varphi=\varphi(\tilde r,\alpha)$ such that  
\begin{equation*}
u_1=\d_2\varphi ,\quad u_2=-\d_{1}\varphi
\end{equation*}
and
\begin{equation*}
u_{\tilde r}=\frac{1}{{\tilde r}^2 \sin\alpha}\d_\alpha\varphi ,\quad u_\alpha=-\frac{1}{{\tilde r} \sin\alpha}\d_{\tilde r}\varphi.    
\end{equation*}
Then the pairs 
\begin{equation*}
(\rho,u)=(\eta(\varphi),\d_2\varphi e_{1}-\d_{1}\varphi e_2),
\end{equation*}
and
\begin{equation*}
(\rho,u)=(\eta(\varphi),\frac{1}{{\tilde r}^2 \sin\alpha}\d_\alpha\varphi e_{\tilde r}-\frac{1}{{\tilde r} \sin\alpha}\d_{\tilde r}\varphi e_\alpha+ u_\theta e_\theta)    
\end{equation*}
 satisfy the mass conservation law, since
 \begin{equation*}
\di(\rho u)=\d_1\rho\d_2\varphi-\d_2\rho\d_1\varphi=0
\end{equation*}
and
\[\di(\rho u)=\d_{\tilde r}\rho \frac{1}{{\tilde r}^2\sin\alpha}\d_\alpha\varphi-\frac{1}{{\tilde r}}\d_\alpha\rho \frac{1}{{\tilde r}\sin\alpha}\d_{\tilde r}\varphi=0.\qedhere\]
\end{proof}
\begin{remark} The stream functions \eqref{CS} and \eqref{SS} are called the Stokes stream functions.
{The above existence results also hold for solutions of the symmetry types I and II with respect to other axis.} We write down the rest of the solutions of symmetry types I and II in the cylindrical and spherical  coordinates:
\begin{align*}
&(\rho,u)=(\eta(\varphi),u_re_r-\frac{1}{r}\d_\theta\varphi e_z+\d_z\varphi e_\theta),\\
&\text{$\varphi$ and $u_r$ depending only on  $z$ and $\theta$},\\
\text{or}\quad & (\rho,u)=(\eta(\varphi),\frac{1}{r}\d_\theta\varphi e_r+u_ze_z-\d_r\varphi e_\theta),  \\
&\text{$\varphi$ and $u_z$ depending only on  $r$ and $\theta$},\\
\intertext{and}
&(\rho,u)=(\eta(\varphi),u_{\tilde r}e_{\tilde r}+\frac{1}{\sin\alpha}\d_\theta\varphi e_\alpha-\d_{\alpha}\varphi e_\theta),\\
&\text{$\varphi$ and $u_{\tilde r}$ depending only on  $\alpha$ and $\theta$},\\
\text{or}\quad & (\rho,u)=(\eta(\varphi),\frac{1}{{\tilde r}^2 \sin\alpha}\d_\theta\varphi e_{\tilde r}+u_\alpha e_\alpha-\frac{1}{\tilde r}\d_{\tilde r}\varphi e_\theta),  \\
&\text{$\varphi$ and $u_{\alpha}$ depending only on  $\tilde r$ and $\theta$}.
\end{align*}

\end{remark}

In the following, we formulate explicit solutions of Navier--Stokes equations \eqref{SNS3} of symmetry types I and II in Cartesian coordinates. Similar explicit solutions to the two-dimensional stationary Navier--Stokes equations \eqref{SNS3} were given in \cite{HL20}.
\begin{itemize}
    \item 
\textit{Symmetry type I:}  
We consider $\rho$ and $u$ depending only on $x_1$ and $x_2$.
Then the system \eqref{SNS3} with $f=0$ is
\begin{equation*}
\label{CC}
\begin{aligned}
\begin{pmatrix}
u_1\d_1 u_1+u_2\d_2u_1\\u_1\d_1 u_2+u_2\d_2u_2\\u_1\d_1 u_3+u_2\d_2u_3
 \end{pmatrix}&-
 \begin{pmatrix}
 2\d_1(\mu \d_1 u_1)+\d_2(\mu(\d_1u_2+\d_2u_1))\\ \d_1(\mu(\d_1u_2+\d_2u_1))+2\d_2(\mu \d_2 u_2)\\
 \d_1(\mu\d_1 u_3)+\d_2(\mu\d_2u_3)
 \end{pmatrix} \\& + \begin{pmatrix}
 \d_1\Pi\\\d_2\Pi\\\d_3\Pi
 \end{pmatrix}=0.
\end{aligned}
\end{equation*}

We also consider $\Pi=\Pi(x_1,x_2)$. Then $(\rho,u_1e_1+u_2e_2)$ satisfies the two-dimensional stationary Navier--Stokes equations \eqref{SNS3}.  
Based on the radial solutions of the two-dimensional equations \eqref{SNS3} given in \cite{HL20}, we obtain a solution that satisfies
\begin{gather*}
 \rho=\rho(r),\quad (r,\theta)=(\sqrt{x_1^2+x_2^2},\arctan(x_2/x_1)),\\
  u_1=(rg\sin\theta),\quad u_2=-(rg\sin\theta),\quad \d_r(\mu\d_r u_3)=0,
\end{gather*}
where $g$ satisfies the ODE
\begin{equation*}
    \d_r( \mu r^3 \d_r g) = -Cr,\quad C\in\R.
\end{equation*}
The corresponding stream function $\Phi$ satisfies
\begin{equation*}
    \d_{rr}(\mu r^3\d_r(\frac1r\d_r\Phi))=-C.
\end{equation*}

\item\textit{Symmetry type II:} We assume that
\begin{equation*}
    (\rho,u)=(\rho(x_3),u_1(x_3)e_1+u_2(x_3)e_2),
\end{equation*}
then the equation  \eqref{SNS3} with $f=0$ reads as
\begin{equation*}
 \begin{pmatrix}
 \d_3(\mu \d_3 u_1)\\ \d_3(\mu \d_3 u_2)\\0
 \end{pmatrix}   = \begin{pmatrix}
 \d_1\Pi\\\d_2\Pi\\\d_3\Pi
 \end{pmatrix} .
\end{equation*}
The equation $\d_3\Pi=0$ implies that $\Pi$ is independent of $x_3$. It follows that there exists constants $C_1,\,C_2\in \R$ such that  
\begin{equation*}
    \d_1\Pi=\d_3(\mu \d_3 u_1)=C_1,\quad  \d_2\Pi=\d_3(\mu \d_3 u_2)=C_2.
\end{equation*}
\end{itemize}

\section{Proof of existence}\label{subsec:proof}

We sketch the proof of Theorem~\ref{thm:R3}, leaving out some details that are already contained in \cite{HL20}.

 Let $u\in H^1_\sigma(\Omega)$ be an axially symmetric divergence-free vector field. We define the projection domain $\tilde\Omega:=(0,r_1)\times (z_1,z_2)$. Notice that we have $(-ru_z,ru_r)\in H^1(\tilde\Omega)$, and there exists 
a stream function $\varphi=\varphi(r,z)\in H^2(\tilde \Omega)$ such that 
\begin{equation}
\label{def:phi:2}
  ru_r=\d_z\varphi,\quad ru_z=-\d_r \varphi.
\end{equation}
 Let $u|_{\d\Omega}=u_0=u_{0,r}e_r+u_{0,\theta}e_\theta+u_{0,z}e_z$ denote the trace of $u$. We define $\tilde\Gamma:=\{r_1\}\times[z_1,z_2]\subset \d\tilde\Omega$ with the unit normal and tangential vectors $\tilde n=(1,0)$ and $\tilde \tau=(0,1)$. Then $\varphi$ satisfies the boundary conditions
\begin{align*}
    \frac{\d\varphi}{\d \tilde n}\Big|_{\tilde\Gamma}=-ru_{0,z}\quad\text{and}\quad\frac{\d\varphi}{\d \tilde\tau}\Big|_{\tilde\Gamma}=ru_{0,r}.
\end{align*}
We parameterize $\tilde\Gamma$ by $\gamma(s):[0,1]\to \tilde\gamma$ such that $ \gamma'(s)=\tilde\tau \gamma(s)$. With a fixed constant $C_0$, we have
\begin{equation}  
\label{BC:phi}
\begin{aligned}
&\varphi(\gamma(s))|_{\tilde\Gamma}=\Phi_0(\gamma(s)):=\int_0^sru_{0,r}(\gamma(s))+C_0\\
&\text{and}\quad\frac{\d\varphi(\gamma(s))}{\d \tilde n}\Big|_{\tilde\gamma}=\Phi_1(\gamma(s)):=-r u_{0,z}(\gamma(s)).
    \end{aligned}
\end{equation}
Throughout this paper we fix $C_0=0$, so that the stream functions $\varphi$ are uniquely determined by \eqref{def:phi:2}-\eqref{BC:phi}.
The stream function $\varphi$ can also be seen as an axially symmetric function defined on $\Omega$. The fact $\sum_{i,j=1}^3|\d_{ij} \varphi |^2\lesssim |\d_{rr}\varphi|^2+|\d_{rz}\varphi|^2+|\d_{zz}\varphi|^2$ ensures that $\varphi\in H^2(\Omega)$. We define $\Gamma=\tilde\Gamma\times[0,2\pi)\subset\d\Omega$, then the boundary value \eqref{BC:phi} implies that
\begin{equation}  
\varphi|_{\Gamma}=\Phi_0\quad\text{and}\quad\frac{\d\varphi}{\d  n}|_{\Gamma}=\Phi_1,
\end{equation}
where the outer normal vector $n$ of $\Gamma$ is defined as $n=(\cos\theta,\sin\theta,0)$.

%We conclude that for any $u\in H^1_\sigma(\Omega)$, there exists a unique stream function $\varphi\in H^2(\Omega)$ satisfying \eqref{BC:phi}.

We define the functional space
\begin{equation*}
   H(\Omega)= H_\sigma^1(\Omega)\cap H^1_0(\Omega).
\end{equation*}
Then for any $u\in H (\Omega)$, there exists a unique stream function $\varphi(r,\theta)\in H^2_{\Gamma}(\Omega)$ such that
\begin{equation*}
u=\frac{1}{r}\d_z\varphi e_r-\frac{1}{r}\d_r\varphi e_z+u_\theta e_\theta  \quad r>0.    
\end{equation*}
In the above, the space $H^2_{\Gamma}(\Omega)$ is defined as $H^2_{\Gamma}(\Omega):=\{\varphi\in H^2(\Omega)\mid \text{$\varphi$ axially symetric, and }\varphi|_{\Gamma}= \frac{\d\varphi}{\d  n}|_{\Gamma}=0\}$.

{We prove Theorem \ref{thm:R3} in the following three steps.}
\subsection*{Step 1: Boundary condition}
We extend the axially symmetric boundary value $u_0$ to $\R^3$ (still denoted by $u_0\in H_\sigma^1(\R^3)$). Then there exists a axially symmetric stream function $\varphi_0\in H^2(\R^3)$ determined by \eqref{BC:phi} such that
\begin{equation*}
    u_0=\frac{1}{r} \d_z\varphi_0 e_r-\frac{1}{r}\d_r\varphi_0e_z+u_{0,\theta}e_\theta,\quad r>0.
\end{equation*}
Recall the axially symmetric domain $\Omega=[0,r_1)\times(z_1,z_2)\times[0,2\pi)$. We define a smooth cut-off function $\zeta(r,z;\delta)$ such that
\begin{equation*}
\zeta(r,z;\delta)=
\begin{cases*}
1 & if $\min\{|r-r_1|,|z-z_i|\}\le\frac{\delta}{2}$,\\
0 & if $\min\{|r-r_1|,|z-z_i|\}\ge \delta$
\end{cases*}
\end{equation*}    
for $i=1,2$ and 
    $$|\zeta(r,z;\delta)|\le C,\quad|\nabla \zeta(r,z;\delta)|\le C\delta^{-1}$$
for some constant $C$. We write
\begin{equation*}
\varphi_0^\delta(r,z)=\varphi_0(r,z)\zeta(r,z;\delta),\quad u_0^\delta=\frac{1}{r} \d_z\varphi_0^\delta e_r-\frac{1}{r}\d_r\varphi_0^\delta e_z+u_{0,\theta}e_\theta.
\end{equation*}
For a fixed $\delta>0$, we only need to find weak solutions
     \begin{equation*}
     u^\delta= u-u_0^\delta\in H(\Omega)
     \end{equation*}
satisfying
\begin{equation}
\label{ueq1}
\begin{split}
\frac12\int_{\Omega}\mu^\delta
Su^\delta:Sv\,dx
=&\int_{\Omega}\rho^\delta(u_0^\delta+u^\delta)
\otimes(u_0^\delta+u^\delta)):\nabla v\,dx
+\int_{\Omega}f\cdot v\,dx\\
&-\frac12\int_{\Omega}\mu^\delta Su_0^\delta:Sv\,dx
\quad \text{for all } v\in H(\Omega),
\end{split}
\end{equation} 
where $\rho^\delta=\eta(\varphi^\delta_0+\varphi^\delta)$, $\mu^\delta=b(\rho^\delta)$ and $\varphi^\delta\in H^2_{\Gamma}(\Omega)$ is a stream function of $u^\delta$.

\subsection*{Step 2: Linearised system}

We fix an axially symmetric $\tilde u\in H(\Omega)$ and the corresponding unique stream function $\tilde\varphi\in H^2_{\Gamma}(\Omega)$. Correspondingly, we define the density function and viscosity coefficient as
    \begin{equation*}
\tilde\rho^\delta=\eta(\tilde \varphi+\varphi_0^\delta),\quad \tilde\mu^\delta=b(\tilde\rho^\delta).
\end{equation*}
We consider the following linearised problem with a parameter $\lambda\in[0,1]$:
\begin{equation}\label{ueq2}
\begin{split}
\frac12\int_{\Omega}\tilde\mu^{\delta}_{\lambda} &S
u:Sv\,dx
=\lambda\int_{\Omega}\tilde\rho^{\delta}_{\lambda}(\lambda u_0^{\delta}+\tilde u)
\otimes(\lambda u_0^{\delta}+u):\nabla v\,dx\\
&+\lambda\int_{\Omega}f\cdot v\,dx
-\frac{\lambda}{2}\int_{\Omega}\tilde\mu^{\delta}_{\lambda} Su_0^{\delta}:Sv\,dx
\quad \text{for all } v\in H(\Omega),
\end{split}
\end{equation}
where $\tilde\rho^{\delta}_{\lambda}=\eta(\tilde \varphi+\lambda\varphi_0^\delta)$ and
$\tilde\mu^{\delta}=b(\tilde\rho^{\delta}_{\lambda})$. Notice that if $\lambda=1$ and $u_1^\delta=\tilde u$, then $u_1^\delta$ satisfies the weak formulation \eqref{ueq1}.

Notice that the left-hand side of \eqref{ueq2} defines an inner product $\langle\cdot,\cdot\rangle$ on $H(\Omega)$ through
\begin{equation*}
\label{product}
\langle u, v\rangle:=\frac12 \int_{\Omega}\tilde \mu^{\delta}_\lambda Su:Sv\,dx.
\end{equation*}
Since by definition $\mu_*\le \tilde \mu^\delta_{\lambda}\le \mu^* $, one has
\begin{align*}
\sqrt{\mu_\ast/2}\|\nabla u\|_{L^2(\Omega)}&\leq \langle u,u\rangle^{\frac12}\leq \sqrt{\mu^\ast/2}\|\nabla u\|_{L^2(\Omega)},
\end{align*} 
which implies the equivalence of $\langle\cdot,\cdot\rangle^{\frac12}$ and $\|\cdot\|_{H^1(\Omega)}$ on $H(\Omega)$.
Furthermore, one can show that the right-hand side of \eqref{ueq2} defines a bounded linear functional. By the Leray--Schauder's principle, there exists a unique weak solution $u^\delta_\lambda \in H(\Omega)$ of the linear problem \eqref{ueq2}. 

\subsection*{Step 3: Nonlinear problem}

We define the map
\begin{equation}
\label{map:T}
T^{\delta}:[0,1]\times H(\Omega)\ni (\lambda,\tilde u)\mapsto u^{\delta}_\lambda\in H(\Omega).
\end{equation}
One can show the existence of the fixed point $u^\delta_1=T(1,u^\delta_1)$ using the Leray--Schauder principle.
A uniform bound for $\|u^{\delta}_{\lambda}\|_{H^1}$ ($\|\varphi^{\delta}_{\lambda}\|_{H^2}\le C$) with $u^\delta_\lambda=T(\lambda,u^\delta_\lambda)$ can be shown by a contraction argument as in \cite{Leray}. Notice that the fixed point  $u^{\delta}=u_1^{\delta}\in H(\Omega)$ satisfies \eqref{ueq1} and the pair $(\rho^\delta,u^\delta_0+u^\delta)\in L^\infty(\Omega)\times H^1_\delta(\Omega)$ is a weak solution of \eqref{SNS3}. 

Higher regularity of the quantities in \eqref{ueq2} is required to show the existence of fixed points of $T^\delta$.
To overcome this, one can first regularise $u_0^\delta$, $\varphi_0^\delta$, $\eta$ and $b$ in \eqref{ueq2} by convolution with a sequence of mollifiers $\sigma^\varepsilon$. Then we repeat the linear and nonlinear arguments, and pass to the limit by letting $\varepsilon\to0$ to obtain \eqref{ueq1}.

\begin{remark}\label{rem:R3}
We can follow the above proof to show the solvability under the assumption of symmetry type II in Subsection \ref{sys}. In this case,  $\rho$ and $\mu$ are fixed and independent of $u$. As a consequence, we do not need to fix $\tilde u$ as in Step 2. The solvability can be obtained by applying the Leray--Schauder principle.
\end{remark}

\section*{Acknowledgement}
The author would like to thank her Ph.D. supervisor JProf. Xian Liao for many helpful discussions and suggestions. The author is pleased to acknowledge financial support by the Deutsche Forschungsgemeinschaft (DFG, German Research Foundation) -- Project-ID 317210226 -- SFB 1283.
\printbibliography

\end{document}